\documentclass[11pt]{amsart}
\usepackage{amsthm,amsmath,amssymb}

\usepackage{geometry}
\usepackage{graphicx,enumitem}
\usepackage{color}
\usepackage{cite}

\numberwithin{equation}{section}
\numberwithin{figure}{section}
\theoremstyle{plain}
\newtheorem{thm}{Theorem}[section]
\newtheorem{lem}[thm]{Lemma}
\newtheorem{cor}[thm]{Corollary}

\theoremstyle{remark}

\newcommand{\M}{\operatorname{M}}

\begin{document}

\title{Tilted Halved Hexagons: Hexagons, Semi-hexagons, and Halved Hexagons Under One Roof}

\author{Tri Lai}
\address{Department of Mathematics, University of Nebraska -- Lincoln, Lincoln, NE 68588, U.S.A.}
\email{tlai3@unl.edu}
\thanks{This research was supported in part  by Simons Foundation Collaboration Grant (\# 585923).}

\subjclass[2010]{05A15,  05B45}

\keywords{perfect matchings, plane partitions, lozenge tilings}

\date{\today}

\dedicatory{}

\begin{abstract}
We investigate a new family of regions that is the universal generalization of three well-known region families in the field of enumeration of tilings: the quasi-regular hexagons, the semi-hexagons, and the halved hexagons. We prove a simple product formula for the number of tilings of these new regions. Our main result also yields the enumerations of two special classes of plane partitions with restricted parts.
\end{abstract}

\maketitle
\section{Introduction}\label{Sec:Intro}

One of the most well-known regions in the field of  enumeration of tilings is the `\emph{quasi-regular hexagon}',  a centrally symmetric hexagon with all $120^{\circ}$ angles. MacMahon's theorem \cite{Mac} on plane partitions fitting in an $(a\times b\times c)$-box yields an elegant product formula for the number of lozenge tilings of a quasi-regular hexagon $H_{a,b,c}$ of side-lengths $a,b,c,a,b,c$ \cite{DT} (in counter-clockwise order, starting from the north side\footnote{From now on we always list the side-lengths of a region in this order.}): 
\begin{equation}\label{Maceq}
\M(H_{a,b,c})=\prod_{i=1}^{a}\prod_{j=1}^{b}\prod_{k=1}^{c}\frac{i+j+k-1}{i+j+k-2},
\end{equation}
where $\M(R)$ denotes the number of lozenge tilings in the region $R$. Here, a \emph{lozenge} (or \emph{unit rhombus}) is a union of any two unit equilateral triangles sharing an edge, and a \emph{lozenge tiling} of a region is a covering of the region by lozenges with no gaps or overlaps. See Figure \ref{Fig:3polar} (a) for an example of a quasi-regular hexagon and Figure \ref{Fig:3polar}(d) for a lozenge tiling.

\begin{figure}\centering
\includegraphics[width=12cm]{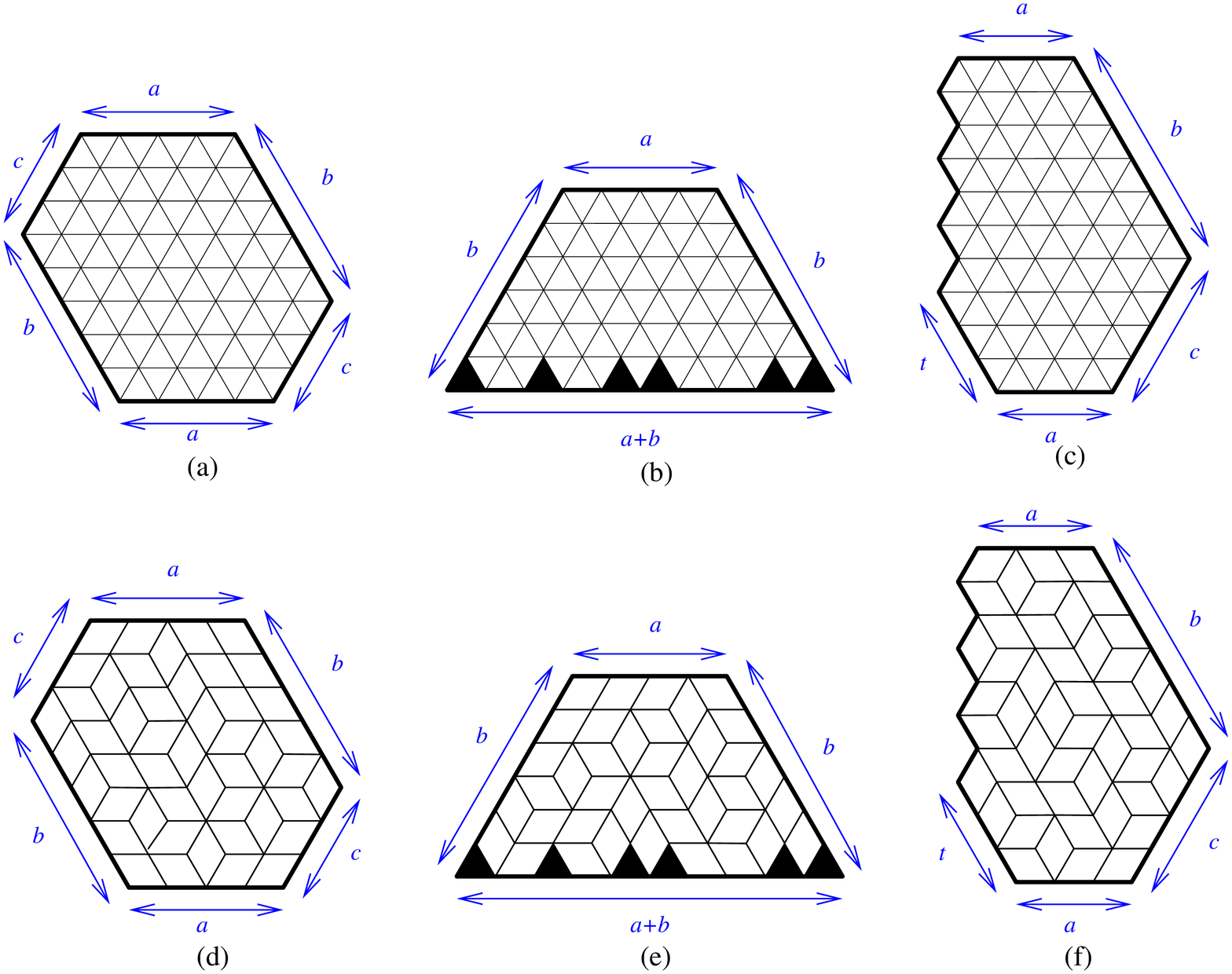}
\caption{(a) A quasi-regular hexagon, (b) a dented semi-hexagon, (c) a halved hexagon, and (d,e,f) their tilings.}\label{Fig:3polar}
\end{figure}

Cohn, Larsen, and Propp \cite[Proposition 2.1]{CLP} generalized MacMahon's theorem by giving a one-to-one correspondence between lozenge tilings of a \emph{semi-hexagon} with dents on the base and semi-strict Gelfand--Tsetlin patterns \cite{GT}. The (dented) semi-hexagon $S_{a,b}(s_1,s_2,\dots,s_b)$ is  a trapezoidal region of side-lengths $a,b,a+b,b$ with  $b$ removed up-pointing unit triangles at the positions $s_1,s_2,\dots,s_b$ on the base, as they appear from left to right. These removed unit triangles are usually called the `\emph{dents}' of the region. See Figure \ref{Fig:3polar} (b)  and (e) for an example of a semi-hexagon and  a lozenge tiling of its; the black triangles indicate the unit triangles removed.
The lozenge tilings of the semi-hexagon $S_{a,b}(s_1,s_2,\dots,s_b)$ are also in bijection with the column-strict plane partitions (or, reverse semi-standard Young tableaux) with shape $(s_b-b,s_{b-1}-b+1,\dots,s_2-2,s_1-1)$. In particular, we have the following tiling formula:
\begin{equation}\label{CLPeq}
\M(S_{a,b}(s_1,s_2,\dots,s_b))=\prod_{1\leq i <j \leq a+b}\frac{s_{j}-s_{i}}{j-i}.
\end{equation}

R. Proctor  \cite[Corollary 4.1]{Proc} enumerated a certain class of staircase plane partitions that, in turn, are in bijection with the lozenge tilings of a quasi-regular hexagon with a maximal staircase cut off $P_{a,b,c}$ ($c\leq b$). We often called this  region  a `\emph{halved hexagon}.' See Figure \ref{Fig:3polar} (c)  for a halved hexagon (it is easy to see that $t=b-c$ in the picture) and Figure \ref{Fig:3polar}(f) for  a lozenge tiling of the halved hexagon. We have
\begin{equation}\label{Proctoreq}
\M(P_{a,b,c})=\prod_{i=1}^{a}\left[\prod_{j=1}^{b-c+1}\frac{a+i+j-1}{i+j-1}\prod_{j=b-c+2}^{b-a+i}\frac{2a+i+j-1}{i+j-1}\right],
\end{equation}
 where empty products are taken to be 1. It is worth noticing that when $b=c$, Proctor's formula gives the enumeration of the transpose-complementary plane partitions, one of the ten symmetry classes of plane partitions \cite{Stanley2}.

The above three tiling enumerations are three of the most popular results in the field of enumeration tilings. They have inspired many other results by various authors. It motivates us to find a universal generalization of the three regions: the quasi-regular hexagons, the semi-hexagons, and the halved hexagons. In this paper, we provide such a generalization and prove an exact tiling enumeration for it.

We consider a quasi-regular hexagon with side-lengths $x,$ $(l-1)k+t,$ $l,$ $x,$ $(l-1)k+t,$ $l$, then cut off a maximal \emph{$k$-staircase} whose $l$ steps have width $k$. When $k=1$, we have exactly a halved hexagon. When $k=0$, there is nothing cut off, the region is still a quasi-regular hexagon. When $k\geq 2$, we have new regions similar to the halved hexagons; however, the cut is tilted. We call the new regions \emph{$k$-halved hexagons} or \emph{tilted halved hexagons}. See Figure \ref{Fig:Tiltinghex} for several examples of the $k$-halved hexagons.

\begin{figure}\centering
\includegraphics[width=10cm]{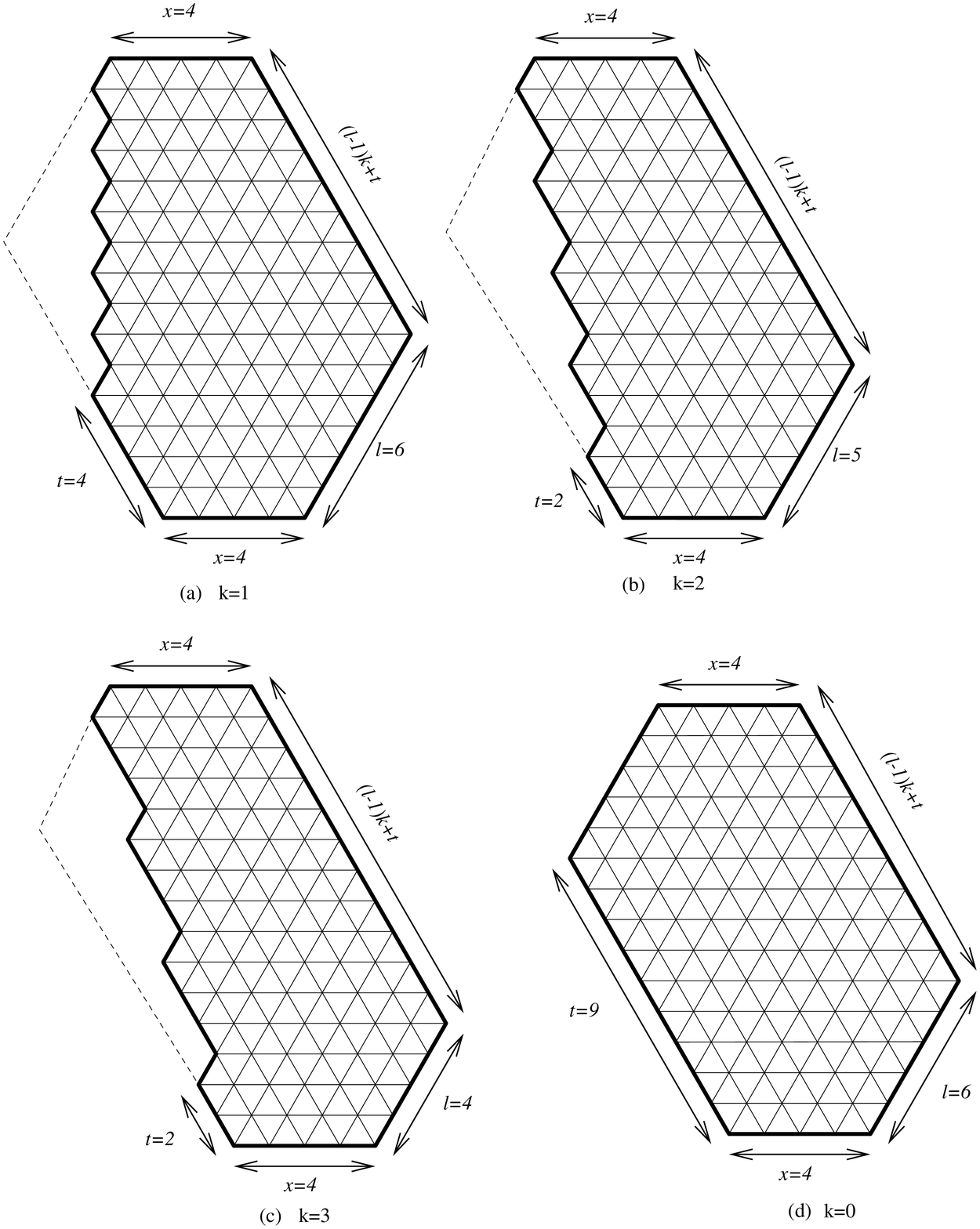}
\caption{The $k$-halved hexagons with no dents.}\label{Fig:Tiltinghex}
\end{figure}

We consider a more general situation when we allow some `dents' to appear on the staircase cut. We now start with a quasi-regular hexagon of side-lengths $x,$ $h(k+1)+(l-1)k+t,$ $l,$ $x,$ $h(k+1)+(l-1)k+t,$ $l$. We cut off a maximal $k$-staircase with $h+l$ steps. Label the staircase levels from the bottom to the top by $1,2,\dots,l+h$, for some non-negative integers $h,l$. We allow removing $h$ up-pointing unit triangles at certain corners of the staircase. Assume that the remaining steps have labels $a_1,a_2,\dots,a_l$ as they appear from bottom to the top. Denote by $H_{x,t,h}(a_1,a_2,\dots,a_l)$ the resulting region (see Figure \ref{Fig:Tiltinghalvehex2} for examples; the black triangles indicate the unit triangles removed). We still call this region a $k$-halved hexagon. When $k=t=x=0$, we get precisely the dented semi-hexagon in Cohn--Larsen--Propp's formula. It is worth noticing that Ciucu has investigated the dented halved hexagon (i.e., the case $k=1$) in \cite{Ciucu2}. However, as the author's knowledge, the cases when $k\geq 2$ have not been considered in the literature. The main theorem of this paper provides an exact tiling formula for the region $H_{x,t,h}(a_1,a_2,\dots,a_l)$.

\begin{figure}\centering
\setlength{\unitlength}{3947sp}%
\begingroup\makeatletter\ifx\SetFigFont\undefined%
\gdef\SetFigFont#1#2#3#4#5{%
  \reset@font\fontsize{#1}{#2pt}%
  \fontfamily{#3}\fontseries{#4}\fontshape{#5}%
  \selectfont}%
\fi\endgroup%
\resizebox{!}{18cm}{
\begin{picture}(0,0)%
\includegraphics{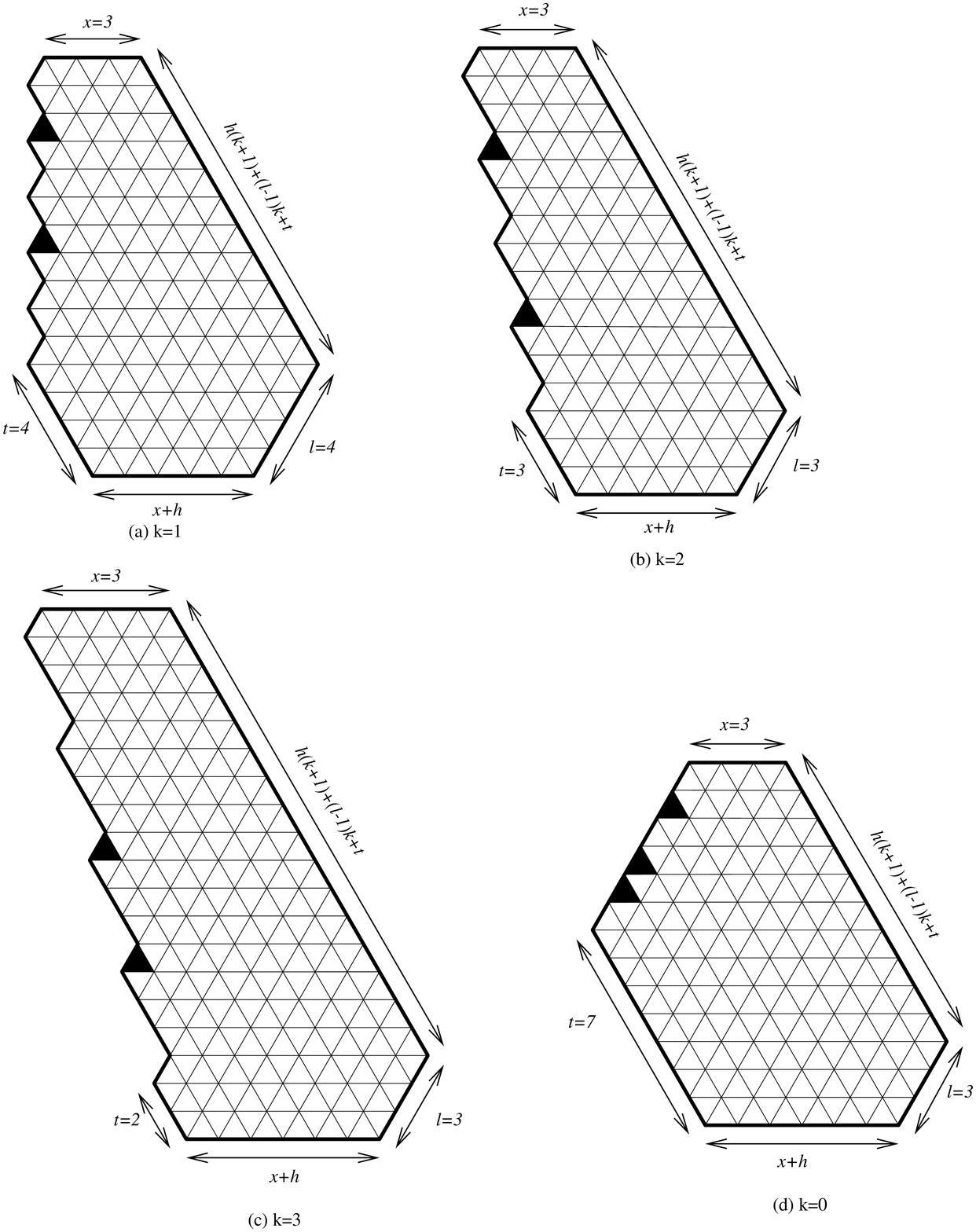}%
\end{picture}%

\begin{picture}(12527,15695)(1225,-15144)
\put(1284,-4179){\makebox(0,0)[lb]{\smash{{\SetFigFont{12}{14.4}{\rmdefault}{\mddefault}{\updefault}{\color[rgb]{0,0,0}$a_1$}%
}}}}
\put(1276,-3466){\makebox(0,0)[lb]{\smash{{\SetFigFont{12}{14.4}{\rmdefault}{\mddefault}{\updefault}{\color[rgb]{0,0,0}$a_2$}%
}}}}
\put(1284,-2041){\makebox(0,0)[lb]{\smash{{\SetFigFont{12}{14.4}{\rmdefault}{\mddefault}{\updefault}{\color[rgb]{0,0,0}$a_3$}%
}}}}
\put(1291,-631){\makebox(0,0)[lb]{\smash{{\SetFigFont{12}{14.4}{\rmdefault}{\mddefault}{\updefault}{\color[rgb]{0,0,0}$a_4$}%
}}}}
\put(7636,-4749){\makebox(0,0)[lb]{\smash{{\SetFigFont{12}{14.4}{\rmdefault}{\mddefault}{\updefault}{\color[rgb]{0,0,0}$a_1$}%
}}}}
\put(2875,-13293){\makebox(0,0)[lb]{\smash{{\SetFigFont{12}{14.4}{\rmdefault}{\mddefault}{\updefault}{\color[rgb]{0,0,0}$a_1$}%
}}}}
\put(8444,-11321){\makebox(0,0)[lb]{\smash{{\SetFigFont{12}{14.4}{\rmdefault}{\mddefault}{\updefault}{\color[rgb]{0,0,0}$a_1$}%
}}}}
\put(7216,-2627){\makebox(0,0)[lb]{\smash{{\SetFigFont{12}{14.4}{\rmdefault}{\mddefault}{\updefault}{\color[rgb]{0,0,0}$a_2$}%
}}}}
\put(9067,-10263){\makebox(0,0)[lb]{\smash{{\SetFigFont{12}{14.4}{\rmdefault}{\mddefault}{\updefault}{\color[rgb]{0,0,0}$a_2$}%
}}}}
\put(1653,-9056){\makebox(0,0)[lb]{\smash{{\SetFigFont{12}{14.4}{\rmdefault}{\mddefault}{\updefault}{\color[rgb]{0,0,0}$a_2$}%
}}}}
\put(9464,-9558){\makebox(0,0)[lb]{\smash{{\SetFigFont{12}{14.4}{\rmdefault}{\mddefault}{\updefault}{\color[rgb]{0,0,0}$a_3$}%
}}}}
\put(6804,-489){\makebox(0,0)[lb]{\smash{{\SetFigFont{12}{14.4}{\rmdefault}{\mddefault}{\updefault}{\color[rgb]{0,0,0}$a_3$}%
}}}}
\put(1240,-7601){\makebox(0,0)[lb]{\smash{{\SetFigFont{12}{14.4}{\rmdefault}{\mddefault}{\updefault}{\color[rgb]{0,0,0}$a_3$}%
}}}}
\end{picture}}
\caption{The $k$-halved hexagons with dents on the staircase cut.}\label{Fig:Tiltinghalvehex2}
\end{figure}

We use the Pochhammer symbol $(x)_n$ in our formula
\begin{equation}\label{poch}
(x)_n:=
\begin{cases}
x(x+1)(x+2)\dotsc(x+n-1) &\text{if $n>0$;}\\
1 &\text{if $n=0$;}\\
\frac{1}{(x-1)(x-2)(x-3)\dotsc(x+n)} &\text{if $n<0$.}
\end{cases}
\end{equation}

We also use the following function:
\begin{align}\label{phieq}
\Phi_{k}((a_i)_{i=1}^{l}; t,x,h)&=\prod_{i=1}^{l}\frac{(x+h+l+1-a_i)_{t+(k+1)a_i-l-k}}{(h+l+1-a_i)_{t+(k+1)a_i-l-k}}\notag\\
&\times \prod_{j=1}^{k-1}\prod_{i=1}^{\lfloor \frac{l+k-j}{k+1}\rfloor}\frac{((k+1)(x+h)+t+k(k+1)i+(j-1)k-(k+1)(k-1))_{j}}{((k+1)h+t+k(k+1)i+(j-1)k-(k+1)(k-1))_{j}}\notag\\
&\times \prod_{i=1}^{\lfloor \frac{l}{k+1}\rfloor}\frac{((k+1)(x+h)+t+k(k+1)i-k+1)_{(k+1)l+k-(k+1)^2i}}{((k+1)h+t+k(k+1)i-k+1)_{(k+1)l+k-(k+1)^2i}},
\end{align}
where empty products are taken to be 1. By definition, we have  $\Phi_{k}((a_i)_{i=1}^{l}; t,0,h)=1$.

We are now ready to state our main theorem:
\begin{thm}\label{mainthm} For non-negative integers $x,t,h,l$ and a sequence  $\textbf{a}=(a_i)_{i=1}^{l}$ of positive integers between $1$ and $h+l$ , we have 
\begin{align}\label{maineq}
\M(H_{x,t,h}(a_1,a_2,\dots,a_l))=\Phi_{k}((a_j)_{j=1}^{l}; t,x,h)\prod_{i=1}^{l}\Phi_{k}((a_j)_{j=1}^{i}; t,a_{i+1}-a_{i}-1,a_{i+1}-i-1),
\end{align}
where $a_{l+1}=l+h+1$ by convention.
\end{thm}

\begin{figure}\centering
\includegraphics[width=8cm]{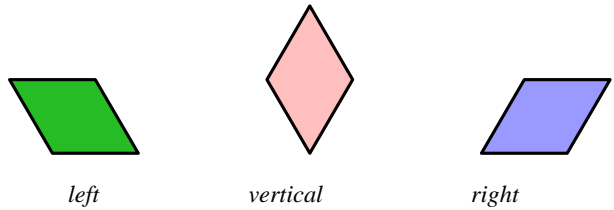}
\caption{Three orientations of the lozenges.}\label{Fig:orientation}
\end{figure}

\begin{figure}\centering
\setlength{\unitlength}{3947sp}%
\begingroup\makeatletter\ifx\SetFigFont\undefined%
\gdef\SetFigFont#1#2#3#4#5{%
  \reset@font\fontsize{#1}{#2pt}%
  \fontfamily{#3}\fontseries{#4}\fontshape{#5}%
  \selectfont}%
\fi\endgroup%
\resizebox{!}{9cm}{
\begin{picture}(0,0)%
\includegraphics{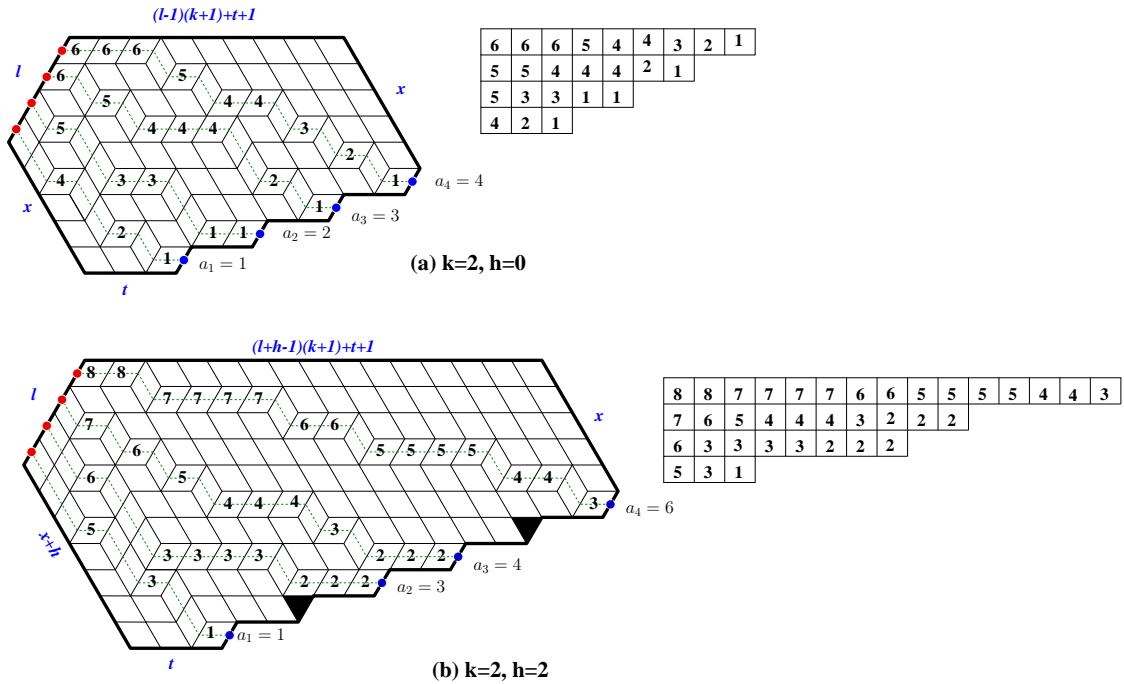}%
\end{picture}%

\begin{picture}(14990,9188)(571,-14414)
\put(3172,-8787){\makebox(0,0)[lb]{\smash{{\SetFigFont{14}{16.8}{\rmdefault}{\bfdefault}{\updefault}{\color[rgb]{0,0,0}$a_1=1$}%
}}}}
\put(4297,-8374){\makebox(0,0)[lb]{\smash{{\SetFigFont{14}{16.8}{\rmdefault}{\bfdefault}{\updefault}{\color[rgb]{0,0,0}$a_2=2$}%
}}}}
\put(5218,-8137){\makebox(0,0)[lb]{\smash{{\SetFigFont{14}{16.8}{\rmdefault}{\bfdefault}{\updefault}{\color[rgb]{0,0,0}$a_3=3$}%
}}}}
\put(6343,-7665){\makebox(0,0)[lb]{\smash{{\SetFigFont{14}{16.8}{\rmdefault}{\bfdefault}{\updefault}{\color[rgb]{0,0,0}$a_4=4$}%
}}}}
\put(3683,-13807){\makebox(0,0)[lb]{\smash{{\SetFigFont{14}{16.8}{\rmdefault}{\bfdefault}{\updefault}{\color[rgb]{0,0,0}$a_1=1$}%
}}}}
\put(5831,-13157){\makebox(0,0)[lb]{\smash{{\SetFigFont{14}{16.8}{\rmdefault}{\bfdefault}{\updefault}{\color[rgb]{0,0,0}$a_2=3$}%
}}}}
\put(6854,-12862){\makebox(0,0)[lb]{\smash{{\SetFigFont{14}{16.8}{\rmdefault}{\bfdefault}{\updefault}{\color[rgb]{0,0,0}$a_3=4$}%
}}}}
\put(8900,-12094){\makebox(0,0)[lb]{\smash{{\SetFigFont{14}{16.8}{\rmdefault}{\bfdefault}{\updefault}{\color[rgb]{0,0,0}$a_4=6$}%
}}}}
\end{picture}}
\caption{The bijection between tilings and partitions in the proof of Corollary \ref{cor1}.}\label{Fig:Tiltinghalfhex5}
\end{figure}

\begin{figure}\centering
\includegraphics[width=10cm]{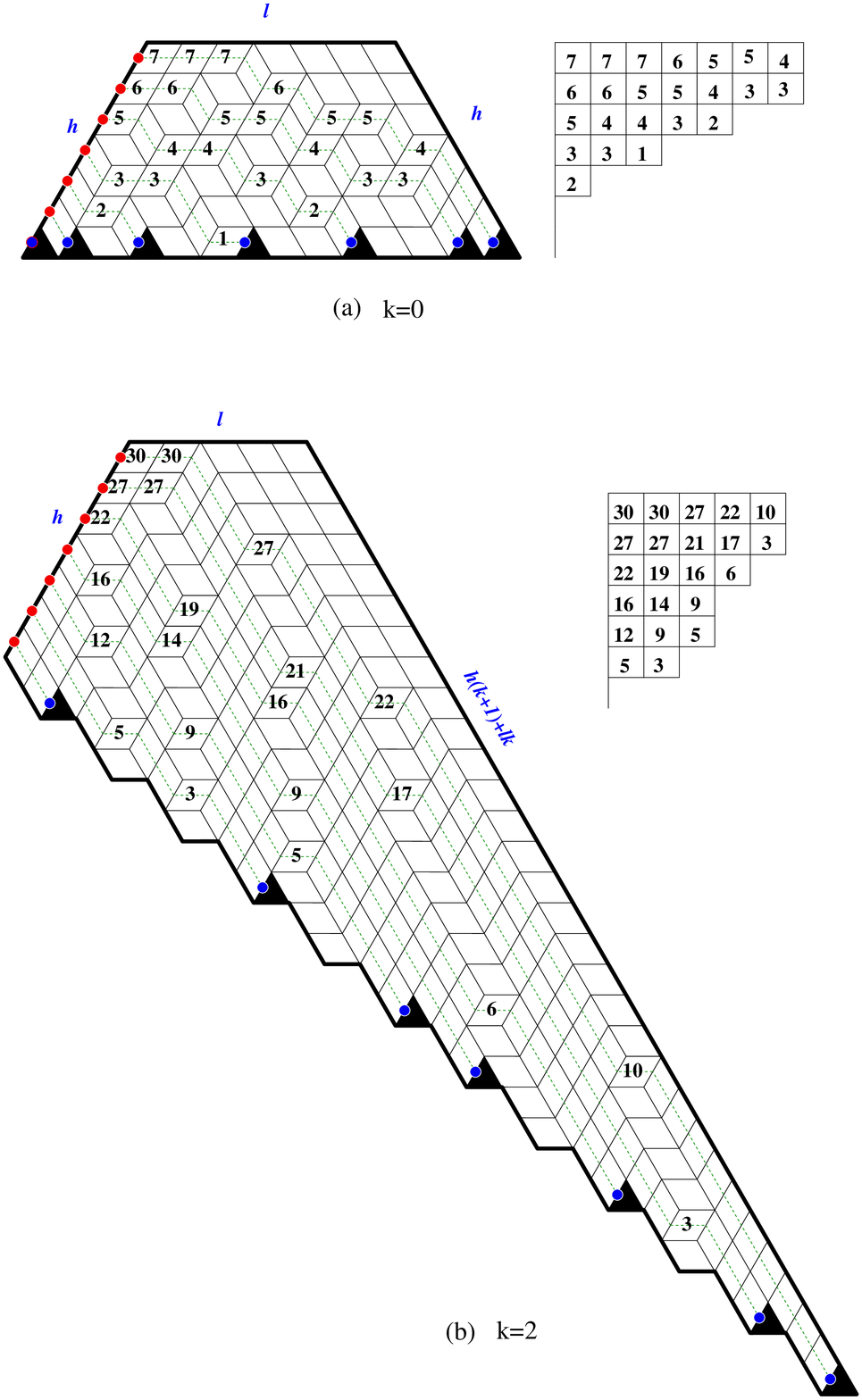}
\caption{The bijection between tilings and partitions in the proof of Corollary \ref{cor2}. The upper plane partition has two rows of length $0$, while the lower partition has one row of length $0$. }\label{Fig:Tiltinghalfhex4}
\end{figure}

In the sake of clarity, we classify the lozenges into three types based on their orientations, as in Figure \ref{Fig:orientation}. In particular, each lozenge tiling may have one of three orientations: \emph{left}, \emph{right}, and \emph{vertical}.

Similarly to the case of MacMahon's, Cohn--Larsen--Propp's, and Proctor's formulas above, our main theorem also implies new enumerative results of certain classes of plane partitions.

\begin{cor}\label{cor1}Assume that $x,t,h,l$ are non-negative integers and  $\textbf{a}=(a_i)_{i=1}^{l}$ is a sequence of positive integers between $1$ and $h+l$. 
Then the number of plane partitions of shape $(t+(k+1)a_l-k-l, t+(k+1)a_{l-1}-k-(l-1),\dots,t+(k+1)a_2-k-2,t+(k+1)a_1-k-1)$, whose parts in the $i$-th row are at least $a_{l-i+1}-l+i$ and at most $x+h$, is equal to
\[\Phi_{k}((a_j)_{j=1}^{l}; t,x,h)\prod_{i=1}^{l}\Phi_{k}((a_j)_{j=1}^{i}; t,a_{i+1}-a_{i}-1,a_{i+1}-i-1).\]
\end{cor}

\begin{proof}
We reflect the region $H_{x,t,h}(a_1,a_2,\dots,a_l)$ over the bisector of its $120^{\circ}$ lower-left angle and get a region like the ones on the left-hand side of Figure \ref{Fig:Tiltinghalfhex5}. See the upper-left picture is for the case $k=2, h=0$, and the lower-left picture for the case $k=2, h=2$. Each tiling of the resulting region can be encoded uniquely as a family of $l$ disjoint lozenge paths consisting of juxtaposing right and vertical lozenges (indicated by the ones with a dotted line in the middle). All lozenges outside these paths are left lozenges. The $i$-th path (ordered from bottom to top) starts at the position $a_i$ on the zigzag side and ends at the position $i$ on the northwest side of the region. We label a right lozenge in the $i$-th path by $i+x$, where $x$ is the number of the vertical lozenges before it as we travel along the path from right to left. For example, the first lozenge path on the upper-left picture of Figure \ref{Fig:Tiltinghalfhex5} has three right lozenges labeled by $1,2,4$, from right to left; the first lozenge path on the lower-left picture has three right lozenges labeled by $1,3,5$. It is easy to see that the $i$-th lozenge path has exactly $t+(k+1)a_i-k-i$ labeled lozenges.

Consider a partition $\lambda$ of shape $(t+(k+1)a_l-k-l, t+(k+1)a_{l-1}-k-(l-1),\dots,t+(k+1)a_2-k-2,t+(k+1)a_1-k-1)$ whose parts in the $j$-th row record the labels of along the $(l-j+1)$-th lozenge path, for $j=1,2,\dots,l$ (see the right pictures in Figure \ref{Fig:Tiltinghalfhex5}). It is easy to see that the resulting plane partition satisfies all the assumptions of the corollary. One could also verify that this mapping is reversible, so is a bijection. The corollary now follows from Theorem \ref{mainthm}.
\end{proof}

Next, we consider the a different class of plane partitions. Our plane partitions  are of shape $\lambda=(\lambda_1,\lambda_2,\dots,\lambda_h)$ (where $\lambda_1\geq \lambda_2\geq \cdots \lambda_n \geq 0$) and satisfy the following three properties:
\begin{enumerate}
\item  the parts are at most $h+k(h+l)$;
\item the $j$-th part from the right in each row is at least $kj$;
\item the different between the $j$-th and $(j+1)$-th parts in each row (if both are defined) is at least $(\lambda_{i}-\lambda_{i+1})k$.
\end{enumerate}

\begin{cor}\label{cor2}
The number of plane partitions of shape $\lambda=(\lambda_1,\lambda_2,\dots,\lambda_h)$ satisfying the properties (1), (2), (3) above
is  equal to 
\[\prod_{i=1}^{l}\Phi_{k}((a_j)_{j=1}^{i}; 0,a_{i+1}-a_{i}-1,a_{i+1}-i-1),\]
where  $\{a_1<a_2<\dots<a_{l}\}$ is the complement of  the set $\{\lambda_{i}+i\}_{i=1}^{h}$, i.e.
\[\{a_1,a_2,\dots,a_l\}=\{1,2,\dots, h+l\}-\{\lambda_{i}+i\}_{i=1}^{h},\]
and where $a_{l+1}=h+l+1$ by convention.
\end{cor}

\begin{proof}
 We reflect the region $H_{0,0,h}(a_1,a_2,\dots,a_l)$ over a vertical line and rotate the resulting region $120^{\circ}$ clockwise. This way, we get a region like the ones  on the left-hand side of Figure \ref{Fig:Tiltinghalfhex4} (see the top picture for the case of $k=0$, and the bottom picture for the case $k=2$). There are $l+h$ horizontal steps on the staircase base of  the region; the positions of the steps with no dent are $a_1,a_2,\dots,a_l$ (ordered from left to right). Assume that  $\{b_1,b_2,\dots,b_h\}$ is the complement of $\{a_1,a_2,\dots,a_l\}$. Its means that $b_1,\dots,b_h$ correspond to the positions of the $h$ dents. We now can encode each tiling of the region by a family of $h$ disjoint lozenge paths consisting of juxtaposing right and vertical lozenges as in the previous corollary. The  paths start at the positions of the dents on the base and end at the northwest side of the region. More precisely, the $j$-th path starts from the $b_j$-dent and contains  exactly $b_j-j$ right lozenges, for $j=1,2,\dots,h$. Each right lozenge in the path is labeled by the number of vertical lozenges before it, as we travel along the path from right to left. For example, in the lower-left picture of Figure \ref{Fig:Tiltinghalfhex4},  the first lozenge path has no right lozenge (as $b_1-1=0$), and the second path has $2$ right lozenges (as $b_2-2=2$) labeled by $3,5$, from right to left. 

Consider the plane partition of shape $\lambda=(b_h-h,b_{h-1}-(h-1),\dots,b_1-1)$, whose $j$-th row records the labels along the $(h-j+1)$-th lozenge path, for $j=1,2,\dots,h$. It is easy to see that the resulting plane partition satisfies properties (1)--(3). This way, we have just created a mapping from the lozenge tilings of the region $H_{0,0,h}(a_1,a_2,\dots,a_l)$ to these plane partitions. It is easy to see that this mapping is a bijection, and the corollary follows from Theorem \ref{mainthm} again.
\end{proof}

It is worth noticing that,  when $k=0$, the proof Corollary \ref{cor2} provides a bijection between lozenge tilings of a semi-hexagon  in Cohn--Larsen--Propp's formula and column-strict plane partitions.




 

\section{Proof of the main theorem}\label{Sec:Proof}

A \emph{forced lozenge} of a tile-able region $R$ is the lozenge that is contained in any tilings. The removal of forced lozenges does not change the number of tilings. A \emph{perfect matching} of a graph is a collection of disjoint edges that cover all vertices of the graph. There is a natural bijection between tilings of a region on the triangular lattice and perfect matchings of its \emph{(planar) dual graph} (the graph whose vertices are the unit triangles in the region and whose edges connect precisely two unit triangles sharing an edge). We use the notation $\M(G)$ for the number of perfect matchings of the graph $G$.

E. H. Kuo \cite{Kuo} proved the following \emph{graphical condensation}, often mentioned as `\emph{Kuo condensation}.' This is the key to our proof.

\begin{lem}[Theorem 2.1 in \cite{Kuo}]\label{Kuolem}
Let $G=(V_1,V_2,E)$ be a plane bipartite graphs in which $|V_1|=|V_2|$. Let vertices $u,$ $v,$ $w,$ and $s$ appear in a cyclic order on a face of $G$. If $u,w \in V_1$ and $v,s\in V_2$, then 
\begin{align}\label{Kuoeq}
\M(G)\M(G-\{u,v,w,s\})=\M(G-\{u,v\})\M(G-\{w,s\})+\M(G-\{u,s\})\M(G-\{v,w\}).
\end{align}
\end{lem}

\begin{figure}\centering
\setlength{\unitlength}{3947sp}%
\begingroup\makeatletter\ifx\SetFigFont\undefined%
\gdef\SetFigFont#1#2#3#4#5{%
  \reset@font\fontsize{#1}{#2pt}%
  \fontfamily{#3}\fontseries{#4}\fontshape{#5}%
  \selectfont}%
\fi\endgroup%
\resizebox{!}{7cm}{
\begin{picture}(0,0)%
\includegraphics{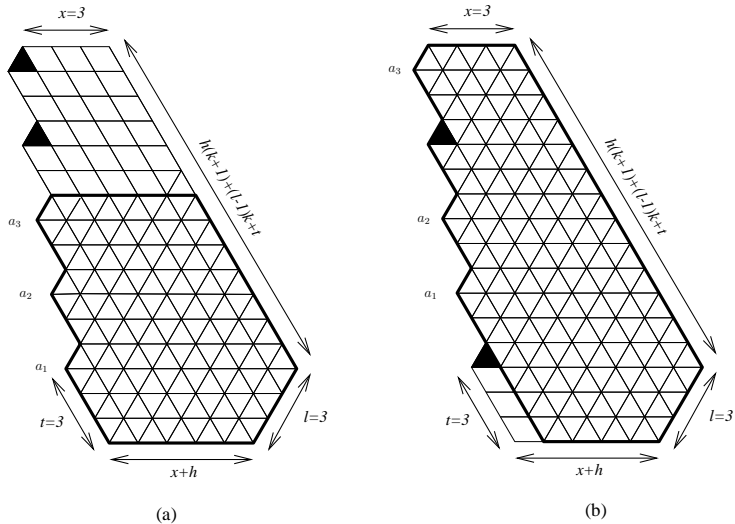}%
\end{picture}%

\begin{picture}(10394,7550)(1623,-7582)
\put(1644,-3264){\makebox(0,0)[lb]{\smash{{\SetFigFont{12}{14.4}{\rmdefault}{\mddefault}{\updefault}{\color[rgb]{0,0,0}$a_3$}%
}}}}
\put(7546,-4299){\makebox(0,0)[lb]{\smash{{\SetFigFont{12}{14.4}{\rmdefault}{\mddefault}{\updefault}{\color[rgb]{0,0,0}$a_1$}%
}}}}
\put(7448,-3214){\makebox(0,0)[lb]{\smash{{\SetFigFont{12}{14.4}{\rmdefault}{\mddefault}{\updefault}{\color[rgb]{0,0,0}$a_2$}%
}}}}
\put(7036,-1076){\makebox(0,0)[lb]{\smash{{\SetFigFont{12}{14.4}{\rmdefault}{\mddefault}{\updefault}{\color[rgb]{0,0,0}$a_3$}%
}}}}
\put(2026,-5342){\makebox(0,0)[lb]{\smash{{\SetFigFont{12}{14.4}{\rmdefault}{\mddefault}{\updefault}{\color[rgb]{0,0,0}$a_1$}%
}}}}
\put(1801,-4299){\makebox(0,0)[lb]{\smash{{\SetFigFont{12}{14.4}{\rmdefault}{\mddefault}{\updefault}{\color[rgb]{0,0,0}$a_2$}%
}}}}
\end{picture}}
\caption{Obtaining a smaller region by removing forced lozenges in the cases (a) $a_l<l+h$ and (b) $a_1>1$.}\label{Fig:Tiltinghalfhex6}
\end{figure}

\begin{figure}\centering
\setlength{\unitlength}{3947sp}%
\begingroup\makeatletter\ifx\SetFigFont\undefined%
\gdef\SetFigFont#1#2#3#4#5{%
  \reset@font\fontsize{#1}{#2pt}%
  \fontfamily{#3}\fontseries{#4}\fontshape{#5}%
  \selectfont}%
\fi\endgroup%
\resizebox{!}{15cm}{
\begin{picture}(0,0)%
\includegraphics{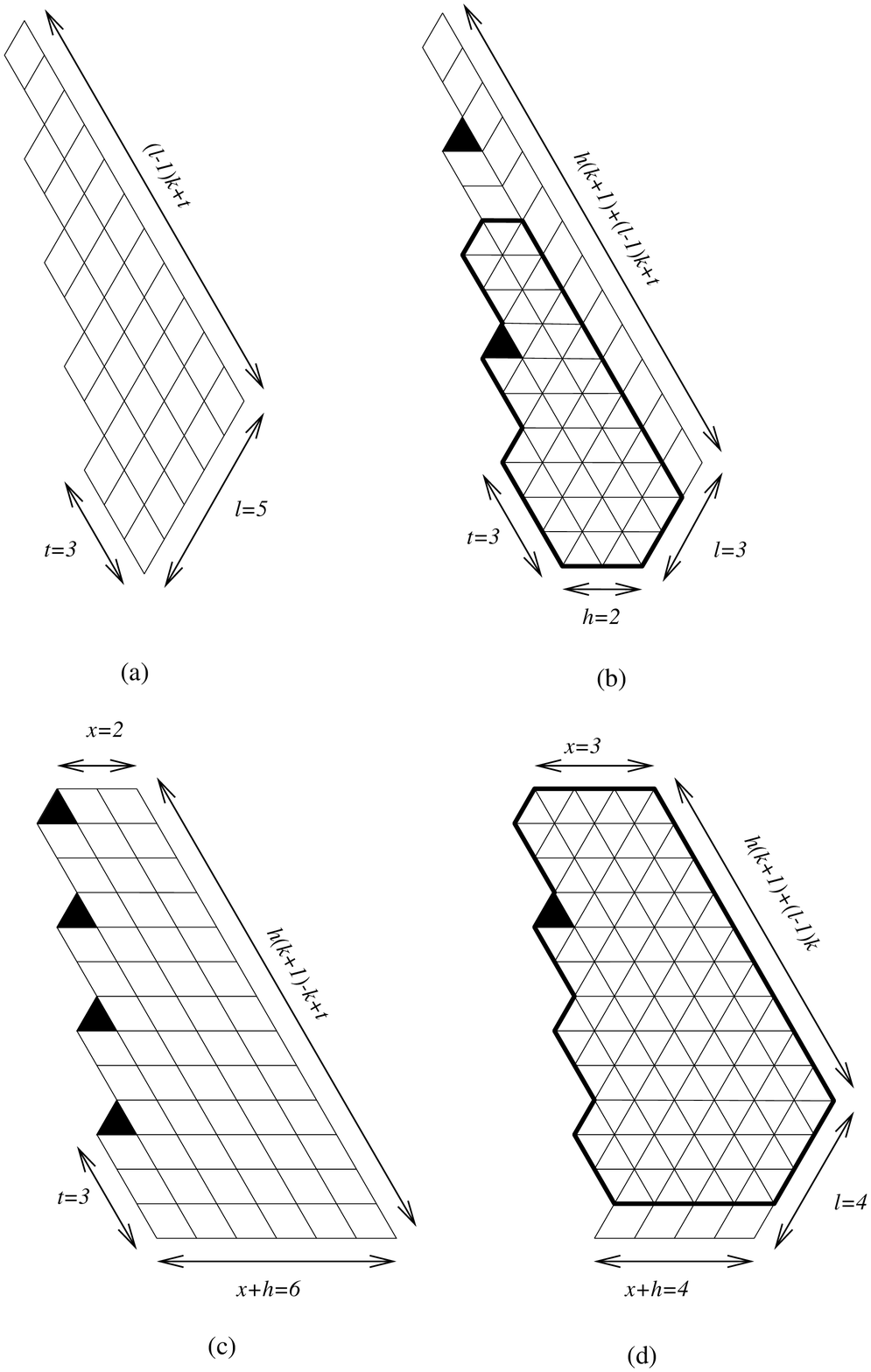}%
\end{picture}%
%
%
\begin{picture}(9329,14051)(814,-14028)
\put(5978,-8511){\makebox(0,0)[lb]{\smash{{\SetFigFont{12}{14.4}{\rmdefault}{\mddefault}{\updefault}{\color[rgb]{0,0,0}$a_4$}%
}}}}
\put(5957,-4737){\makebox(0,0)[lb]{\smash{{\SetFigFont{12}{14.4}{\rmdefault}{\mddefault}{\updefault}{\color[rgb]{0,0,0}$a_1$}%
}}}}
\put(5537,-2615){\makebox(0,0)[lb]{\smash{{\SetFigFont{12}{14.4}{\rmdefault}{\mddefault}{\updefault}{\color[rgb]{0,0,0}$a_2$}%
}}}}
\put(5125,-477){\makebox(0,0)[lb]{\smash{{\SetFigFont{12}{14.4}{\rmdefault}{\mddefault}{\updefault}{\color[rgb]{0,0,0}$a_3$}%
}}}}
\put(1647,-4832){\makebox(0,0)[lb]{\smash{{\SetFigFont{12}{14.4}{\rmdefault}{\mddefault}{\updefault}{\color[rgb]{0,0,0}$a_1$}%
}}}}
\put(1340,-3828){\makebox(0,0)[lb]{\smash{{\SetFigFont{12}{14.4}{\rmdefault}{\mddefault}{\updefault}{\color[rgb]{0,0,0}$a_2$}%
}}}}
\put(1136,-2765){\makebox(0,0)[lb]{\smash{{\SetFigFont{12}{14.4}{\rmdefault}{\mddefault}{\updefault}{\color[rgb]{0,0,0}$a_3$}%
}}}}
\put(931,-1702){\makebox(0,0)[lb]{\smash{{\SetFigFont{12}{14.4}{\rmdefault}{\mddefault}{\updefault}{\color[rgb]{0,0,0}$a_4$}%
}}}}
\put(829,-580){\makebox(0,0)[lb]{\smash{{\SetFigFont{12}{14.4}{\rmdefault}{\mddefault}{\updefault}{\color[rgb]{0,0,0}$a_5$}%
}}}}
\put(6748,-12651){\makebox(0,0)[lb]{\smash{{\SetFigFont{12}{14.4}{\rmdefault}{\mddefault}{\updefault}{\color[rgb]{0,0,0}$a_1$}%
}}}}
\put(6489,-11641){\makebox(0,0)[lb]{\smash{{\SetFigFont{12}{14.4}{\rmdefault}{\mddefault}{\updefault}{\color[rgb]{0,0,0}$a_2$}%
}}}}
\put(6285,-10578){\makebox(0,0)[lb]{\smash{{\SetFigFont{12}{14.4}{\rmdefault}{\mddefault}{\updefault}{\color[rgb]{0,0,0}$a_3$}%
}}}}
\end{picture}}
\caption{Several special cases of the region $H_{x,t,h}(a_1,a_2,\dots,a_l)$.}\label{Fig:Tiltinghalfhex7}
\end{figure}

\begin{figure}\centering
\setlength{\unitlength}{3947sp}%
\begingroup\makeatletter\ifx\SetFigFont\undefined%
\gdef\SetFigFont#1#2#3#4#5{%
  \reset@font\fontsize{#1}{#2pt}%
  \fontfamily{#3}\fontseries{#4}\fontshape{#5}%
  \selectfont}%
\fi\endgroup%

\resizebox{!}{22cm}{
\begin{picture}(0,0)%
\includegraphics{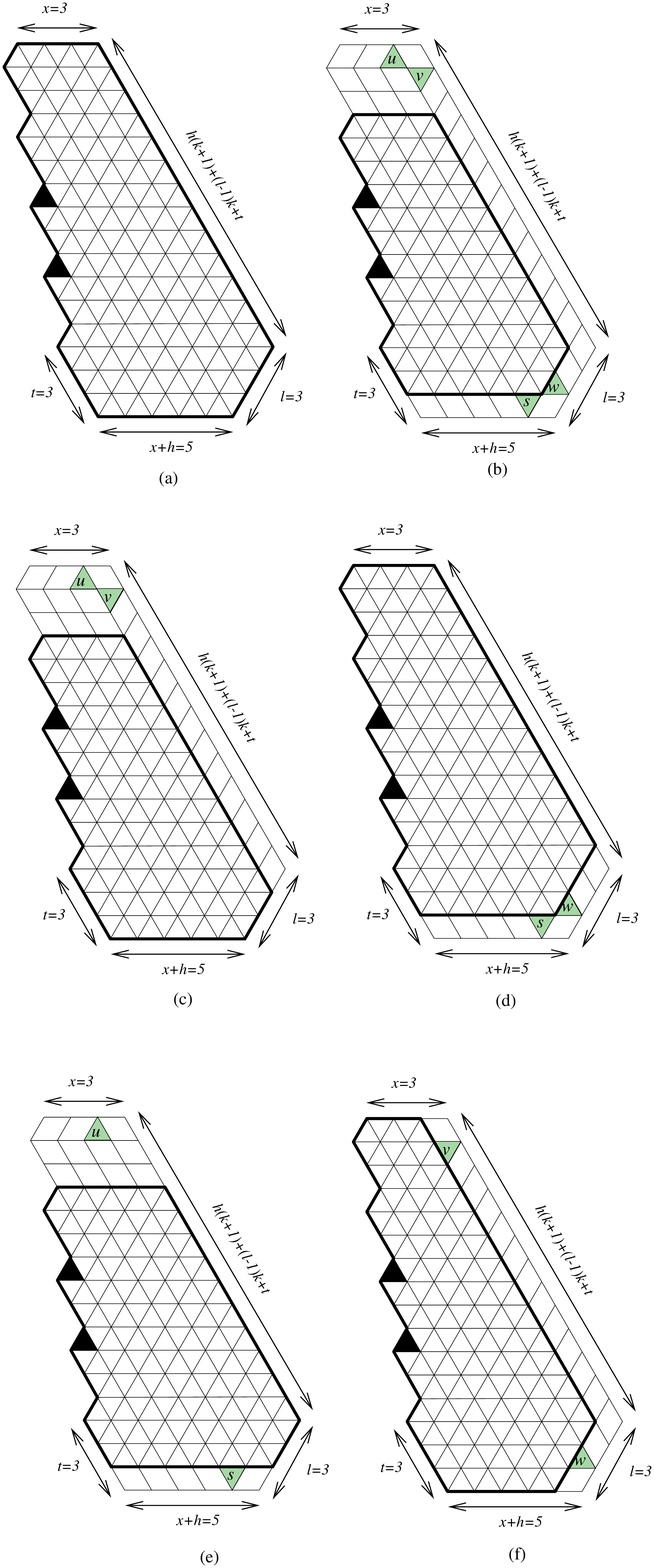}%
\end{picture}%
%
%

\begin{picture}(10375,23895)(1216,-23150)
\put(1996,-4621){\makebox(0,0)[lb]{\smash{{\SetFigFont{14}{16.8}{\familydefault}{\mddefault}{\updefault}{\color[rgb]{0,0,0}$a_1$}%
}}}}
\put(1441,-1449){\makebox(0,0)[lb]{\smash{{\SetFigFont{14}{16.8}{\familydefault}{\mddefault}{\updefault}{\color[rgb]{0,0,0}$a_2$}%
}}}}
\put(1231,-369){\makebox(0,0)[lb]{\smash{{\SetFigFont{14}{16.8}{\familydefault}{\mddefault}{\updefault}{\color[rgb]{0,0,0}$a_3$}%
}}}}
\put(6943,-4611){\makebox(0,0)[lb]{\smash{{\SetFigFont{14}{16.8}{\familydefault}{\mddefault}{\updefault}{\color[rgb]{0,0,0}$a_1$}%
}}}}
\put(2218,-12523){\makebox(0,0)[lb]{\smash{{\SetFigFont{14}{16.8}{\familydefault}{\mddefault}{\updefault}{\color[rgb]{0,0,0}$a_1$}%
}}}}
\put(7153,-12493){\makebox(0,0)[lb]{\smash{{\SetFigFont{14}{16.8}{\familydefault}{\mddefault}{\updefault}{\color[rgb]{0,0,0}$a_1$}%
}}}}
\put(2398,-20916){\makebox(0,0)[lb]{\smash{{\SetFigFont{14}{16.8}{\familydefault}{\mddefault}{\updefault}{\color[rgb]{0,0,0}$a_1$}%
}}}}
\put(7310,-20923){\makebox(0,0)[lb]{\smash{{\SetFigFont{14}{16.8}{\familydefault}{\mddefault}{\updefault}{\color[rgb]{0,0,0}$a_1$}%
}}}}
\put(6350,-1416){\makebox(0,0)[lb]{\smash{{\SetFigFont{14}{16.8}{\familydefault}{\mddefault}{\updefault}{\color[rgb]{0,0,0}$a_2$}%
}}}}
\put(1625,-9343){\makebox(0,0)[lb]{\smash{{\SetFigFont{14}{16.8}{\familydefault}{\mddefault}{\updefault}{\color[rgb]{0,0,0}$a_2$}%
}}}}
\put(6523,-9358){\makebox(0,0)[lb]{\smash{{\SetFigFont{14}{16.8}{\familydefault}{\mddefault}{\updefault}{\color[rgb]{0,0,0}$a_2$}%
}}}}
\put(1835,-17743){\makebox(0,0)[lb]{\smash{{\SetFigFont{14}{16.8}{\familydefault}{\mddefault}{\updefault}{\color[rgb]{0,0,0}$a_2$}%
}}}}
\put(6748,-17728){\makebox(0,0)[lb]{\smash{{\SetFigFont{14}{16.8}{\familydefault}{\mddefault}{\updefault}{\color[rgb]{0,0,0}$a_2$}%
}}}}
\put(6140,-366){\makebox(0,0)[lb]{\smash{{\SetFigFont{14}{16.8}{\familydefault}{\mddefault}{\updefault}{\color[rgb]{0,0,0}$a_3$}%
}}}}
\put(1423,-8271){\makebox(0,0)[lb]{\smash{{\SetFigFont{14}{16.8}{\familydefault}{\mddefault}{\updefault}{\color[rgb]{0,0,0}$a_3$}%
}}}}
\put(6328,-8263){\makebox(0,0)[lb]{\smash{{\SetFigFont{14}{16.8}{\familydefault}{\mddefault}{\updefault}{\color[rgb]{0,0,0}$a_3$}%
}}}}
\put(6552,-16659){\makebox(0,0)[lb]{\smash{{\SetFigFont{14}{16.8}{\familydefault}{\mddefault}{\updefault}{\color[rgb]{0,0,0}$a_3$}%
}}}}
\put(1622,-16659){\makebox(0,0)[lb]{\smash{{\SetFigFont{14}{16.8}{\familydefault}{\mddefault}{\updefault}{\color[rgb]{0,0,0}$a_3$}%
}}}}
\end{picture}}
\caption{Using Kuo condensation to obtain a recurrence for tiling numbers of the $H$-type regions.}\label{Fig:Tiltinghalfhex8}
\end{figure}

We are now ready to prove our main theorem.

\begin{proof}[Proof of Theorem \ref{mainthm}]
Fix some nonnegative integer $k$. We prove identity (\ref{maineq}) by induction on the statistic $p:=x+t+l+(k+2)h$ of the region $H_{x,t,h}(a_1,a_2,\dots,a_l)$. In the rest of this proof, we say that an $H$-type region $A$ is ``smaller" than an $H$-type region $B$ if the $p$-statistic of $A$ is smaller than the $p$-statistic of $B$.

The base case is the case when $l =0$. If $l=0$, identity  (\ref{maineq}) becomes ``$1=1$," as our region has exactly one tiling (see Figure \ref{Fig:Tiltinghalfhex7}(c)), and the expression on the right-hand side of (\ref{maineq}) is precisely $1$. 

For induction step, we assume that $l>0$ and that identity  (\ref{maineq}) holds for any $H$-type region with $p$-statistic strictly smaller than $x+t+l+(k+2)h$. 

We can assume that $a_1=1$ and $a_{l}=h+l$. Indeed, if $a_l<h+l$, then we can remove forced lozenges in the $(k+1)(l+h-a_{l})$ top rows of the region $H_{x,t,h}(a_1,a_2,\dots,a_l)$ to get a smaller $H$-type region, namely $H_{x+l+h-a_{l}, t, a_{l}-l}(a_1,a_2,\dots,a_l)$   (see the region restricted by the bold contour in Figure \ref{Fig:Tiltinghalfhex6}(a)).  If $a_1>1$, we can also get a smaller $H$-type region as in Figure \ref{Fig:Tiltinghalfhex6}(b). The new region is precisely $H_{x, t+(k+1)(a_1-1), h+1-a_1}(1,a_2-a_1+1,a_3-a_1+1,\dots,a_l-a_1+1)$. Then  (\ref{maineq}) holds by the induction hypothesis, when $a_l<l+h$ or $a_1>1$.

Next, if $x=h=0$, then our identity  (\ref{maineq}) becomes ``$1=1$." Our region has only one tiling (see Figure \ref{Fig:Tiltinghalfhex7}(a)), and our tiling formula is simply to $1$.  If $x=0$ and $h>0$, we obtain a smaller region by removing forced lozenges as in Figure \ref{Fig:Tiltinghalfhex7}(b). If $t=0$ (we are assuming that $a_1=1$), then we can also remove forced lozenges along the base to get a smaller region (see Figure \ref{Fig:Tiltinghalfhex7}(d)). It means that (\ref{maineq}) holds in the situations when at least one of $x$ and $t$ is equal to $0$.

\medskip

In the rest of the proof, we assume that  $x,l,t>0,$  and that $a_1=1,a_{l}=h+l$. 

We now apply Kuo's Lemma \ref{Kuolem} to the dual graph $G$ of the region $H=H_{x,t,h}(a_1,a_2,\dots,a_l)$ with the four vertices $u,v,w,s$ chosen as in Figure \ref{Fig:Tiltinghalfhex8}(b) (the four vertices are  indicated by the corresponding unit triangles of the same label in $H$). More precisely,  the $u$- and $v$-triangles are the shaded unit triangles at the upper-right corner of the region, and the $w$- and $s$-triangles are the shaded unit triangles at lower-right corner. We get the recurrence
\begin{align}\label{Grecurrence}
\M(G)\M(G-\{u,v,w,s\})=\M(G-\{u,v\})\M(G-\{w,s\})+\M(G-\{u,s\})\M(G-\{v,w\}).
\end{align}

Next, we will convert recurrence (\ref{Grecurrence}) into a recurrence for the tiling numbers of $H$-type regions. To do so, we will write each matching number in (\ref{Grecurrence}) as the tiling number of some $H$-type region. The task for the first matching number is obvious. By definition, we have:
\begin{equation}\label{eq1}
\M(G)=\M(H_{x,t,h}(a_1,a_2,\dots,a_l)).
\end{equation}

Let us consider the second matching number, i.e., $\M(G-\{u,v,w,s\})$. We consider the region  corresponding to the graph $G-\{u,v,w,s\}$ as shown in Figure \ref{Fig:Tiltinghalfhex8}(b). The removal of the $u$-, $v$-, $w$-, and $s$-triangles yields  some forced lozenges on the top, the northeast, the southeast, and the base of the region. By removing these forced lozenges, we get a new $H$-type region, namely $H_{x,t-1,h}(a_1,\dots,a_{l-1})$ (see the region restricted by the bold contour). As the removal of forced lozenges does not  change the tiling number, we get
\begin{equation}\label{eq2}
\M(G-\{u,v,w,s\})=\M(H_{x,t-1,h}(a_1,\dots,a_{l-1})).
\end{equation}
By considering forced lozenges, as in Figures \ref{Fig:Tiltinghalfhex8}(c)--(f), we get respectively the following identities:
\begin{equation}\label{eq3}
\M(G-\{u,v\})=\M(H_{x,t,h}(a_1,\dots,a_{l-1})),
\end{equation}
\begin{equation}\label{eq4}
\M(G-\{w,s\})=\M(H_{x,t-1,h}(a_1,\dots,a_{l})),
\end{equation}
\begin{equation}\label{eq5}
\M(G-\{u,s\})=\M(H_{x+1,t-1,h}(a_1,\dots,a_{l-1})),
\end{equation}
\begin{equation}\label{eq6}
\M(G-\{v,w\})=\M(H_{x-1,t,h}(a_1,\dots,a_{l})).
\end{equation}

By (\ref{eq1})--(\ref{eq6}), we can convert recurrence (\ref{Grecurrence}) into the following recurrence:
\begin{align}
\M(H_{x,t,h}(a_1,\dots,a_{l}))\M(H_{x,t-1,h}(a_1,\dots,a_{l-1}))&=\M(H_{x,t,h}(a_1,\dots,a_{l-1}))\M(H_{x,t-1,h}(a_1,\dots,a_{l}))\notag\\
&+\M(H_{x+1,t-1,h}(a_1,\dots,a_{l-1}))\M(H_{x-1,t,h}(a_1,\dots,a_{l})).
\end{align}

Our remaining  job is to show that the tiling formula on the right-hand side of (\ref{maineq}) also satisfies this recurrence. Then the theorem follows from the induction principle. In particular, if we denote by $f_{x,t,h}(a_1,\dots,a_{l-1})$ this tiling formula, we need to verify that 
\begin{align}\label{Feq1}
f_{x,t,h}(a_1,\dots,a_{l})f_{x,t-1,h}(a_1,\dots,a_{l-1})&=f_{x,t,h}(a_1,\dots,a_{l-1})f_{x,t-1,h}(a_1,\dots,a_{l})\notag\\
&+f_{x+1,t-1,h}(a_1,\dots,a_{l-1})f_{x-1,t,h}(a_1,\dots,a_{l}).
\end{align}
Equivalently
\begin{align}\label{Feq2}
\frac{f_{x,t,h}(a_1,\dots,a_{l-1})}{f_{x,t,h}(a_1,\dots,a_{l})}\cdot \frac{f_{x,t-1,h}(a_1,\dots,a_{l})}{f_{x,t-1,h}(a_1,\dots,a_{l-1})}+\frac{f_{x+1,t-1,h}(a_1,\dots,a_{l-1})}{f_{x,t-1,h}(a_1,\dots,a_{l-1})}\frac{f_{x-1,t,h}(a_1,\dots,a_{l})}{f_{x,t,h}(a_1,\dots,a_{l})}=1.
\end{align}
Consider the first term on the left-hand side of (\ref{Feq2}). All the $\Phi$-factors in this term cancel out, except for the ones involving $x$:
\begin{align}\label{Feq3}
\frac{f_{x,t,h}(a_1,\dots,a_{l-1})}{f_{x,t,h}(a_1,\dots,a_{l})}\cdot \frac{f_{x,t-1,h}(a_1,\dots,a_{l})}{f_{x,t-1,h}(a_1,\dots,a_{l-1})}=\frac{\Phi_{k}((a_j)_{j=1}^{l-1}; t,x,h)\Phi_{k}((a_j)_{j=1}^{l}; t-1,x,h)}{\Phi_{k}((a_j)_{j=1}^{l-1}; t-1,x,h)\Phi_{k}((a_j)_{j=1}^{l}; t,x,h)}.
\end{align}
 By definition of the $\Phi$-function, one could simplify the right-hand side of (\ref{Feq3}) to
\begin{align}
\frac{\Big((k+1)(l+h-1)+t\Big)\Big((k+1)(x+h+l-1)+t)\Big)}{\Big((k+1)(x+l+h-1)+t\Big)\Big(x+t+k(l-1)+(k+1)h\Big)}.
\end{align}
It means that the first term on left-hand side of  (\ref{Feq2}) reduces to
\begin{align}\label{Feq4}
\frac{f_{x,t,h}(a_1,\dots,a_{l-1})}{f_{x,t,h}(a_1,\dots,a_{l})}\cdot \frac{f_{x,t-1,h}(a_1,\dots,a_{l})}{f_{x,t-1,h}(a_1,\dots,a_{l-1})}=\frac{\Big((k+1)(l+h-1)+t\Big)\Big((k+1)(x+h+l-1)+t)\Big)}{\Big((k+1)(x+l+h-1)+t\Big)\Big(x+t+k(l-1)+(k+1)h\Big)}.
\end{align}

Similarly, one could simply the the second term on left-hand side of (\ref{Feq2}) as
\begin{align}\label{Feq5}
\frac{f_{x+1,t-1,h}(a_1,\dots,a_{l-1})}{f_{x,t-1,h}(a_1,\dots,a_{l-1})}\frac{f_{x-1,t,h}(a_1,\dots,a_{l})}{f_{x,t,h}(a_1,\dots,a_{l})}&=\frac{\Phi_{k}((a_j)_{j=1}^{l-1}; t-1,x+1,h)\Phi_{k}((a_j)_{j=1}^{l}; t,x-1,h)}{\Phi_{k}((a_j)_{j=1}^{l-1}; t-1,x,h)\Phi_{k}((a_j)_{j=1}^{l}; t,x,h)}\\
&=\frac{x\Big((k+1)x+(k+1)h+t\Big)}{\Big((k+1)(x+l+h-1)+t\Big)\Big(x+t+k(l-1)+(k+1)h\Big)}.
\end{align}
The identity  (\ref{Feq2}) now becomes
\begin{align}
&\frac{\Big((k+1)(l+h-1)+t\Big)\Big((k+1)(x+h+l-1)+t)\Big)}{\Big((k+1)(x+l+h-1)+t\Big)\Big(x+t+k(l-1)+(k+1)h\Big)}\notag\\
&\quad\quad\quad\quad\quad\quad\quad\quad\quad\quad\quad+\frac{x\Big((k+1)x+(k+1)h+t\Big)}{\Big((k+1)(x+l+h-1)+t\Big)\Big(x+t+k(l-1)+(k+1)h\Big)}=1,
\end{align}
which is a true identity. This finishes the proof of our theorem. 
\end{proof}

\section{Several open questions}\label{Sec:Q}
We conclude this paper by a couple of open questions.

\medskip

As there exist elegant $q$-analogs of MacMahon's tiling formula (also by MacMahon \cite{Mac}), Cohn--Larsen--Propp's formula (in the language of column-strict plane partitions, see, e.g., \cite[pp. 374--375]{Stanley}), and Proctor formula (see \cite{LR20}), one should expect for a $q$-analog of our main theorem, at least for the case with no dents.
\
\\

\textbf{Open Problem 1.} Find a $q$-analog of Theorem \ref{maineq}.
\
\\

It is well-known that the weighted enumerations of tilings of a quasi-regular hexagon and a semi-hexagon can be written as a Schur polynomial. The weighted enumeration of tilings of the halved hexagon can also be written in terms of the symplectic characters (see, e.g.,  \cite{AF}).
\
\\

\textbf{Open Question 2.}  Is there any symmetric function that counts the weighted tilings of the $k$-halved hexagons?

\bibliographystyle{plain}
\bibliography{Tiltinghex}

\end{document}